\documentclass[12pt]{article}
\usepackage{listings}
\usepackage{url}

\usepackage[english]{babel}
\usepackage{setspace}
\usepackage[letterpaper,top=2cm,bottom=2cm,left=3cm,right=3cm,marginparwidth=1.75cm]{geometry}
\usepackage[dvipsnames,table,xcdraw]{xcolor}

\providecommand{\keywords}[1] 
{
  \small	
  \textit{\textbf{Keywords}}---Periodicity, Distance, Cliques, Connection, Graph Operators
}

\usepackage{amsmath}
\usepackage{graphicx}
\usepackage[colorlinks=true, allcolors=blue]{hyperref}
\usepackage{amsfonts}
\usepackage{amsthm}
\usepackage{amssymb}

\newtheorem{lemma}{Lemma}
\newtheorem{theorem}{Theorem}
\newtheorem{proposition}{Proposition}
\newtheorem{corollary}{Corollary}
\newtheorem{definition}{Definition}

\newtheorem{conjecture}{Conjecture}

\usepackage{graphicx}
\usepackage{qtree}
\usepackage{amsmath}
\usepackage{amsthm}
\usepackage{amssymb}
\usepackage[dvipsnames]{xcolor}
\usepackage{tikz,tkz-euclide}
\usetikzlibrary{decorations.pathreplacing,angles,quotes}
\usetikzlibrary{calc}
\usetikzlibrary{snakes}
\usetikzlibrary{arrows}
\usetikzlibrary{arrows.meta}
\usetikzlibrary{positioning}
\usepackage{mathtools}
\usepackage{algpseudocode}
\usepackage{mathrsfs}
\usepackage{xpatch}
\usetikzlibrary{nfold}
\usepackage[font={footnotesize,stretch=1.3}]{caption}

\usepackage{hyperref} 
\hypersetup{
    colorlinks,
    citecolor=blue,
    filecolor=black,
    linkcolor=blue,
    urlcolor=blue
}

\title{On the Periodicity of $k$-distance Graphs}

\author{Gregory Demo$^{\dag}$, David Lund$^{\dag}$, Oleksiy Al-saadi$^{\ddagger}$ \\
        \small $^{\dag}$Department of Computer Science, Sonoma State University, {\nolinkurl{[demog,lundev]@sonoma.edu}} \\
        \small $^{\ddagger}$Department of Computer Science, Chico State University, {\nolinkurl{oal-saadi@csuchico.edu}} \\
}

\begin{document}

\date{}
\maketitle

\begin{abstract}
The $k$-distance graph $G^k$ of a graph $G$ has the same vertex set as $G$ and two vertices are adjacent if and only if their distance is $k$ in $G$. These graphs have been extensively studied for their connection properties. In this paper, we study various properties of these graphs, including clique number and periodicity. For all $k$, we show that there exists a $k$-distance graph that is weakly periodic for any size period. Our main result is a proof that if $k > 2$ then there exists a $k$-distance graph of any size (strong) period. Finally, we provide evidence that there exists an upper bound on the periodicity of a connected $2$-distance graph as a function of $|V|$. 
\end{abstract}



\section{Introduction}  

Given a graph $G = (V,E)$, we let $V(G)$ and $E(G)$ denote the vertex and edge sets of $G$, respectively. The $k$-distance graph $H$ of any graph $G$ has $V(G) = V(H)$ where any pair is adjacent if and only if they have distance $k$ in $G$. The class of $k$-distance graphs has been well-studied since their introduction by Harary et al$.$ \cite{Harary} as \emph{exact distance powers}. Since then, many researchers have studied a variety of properties of the $k$-distance graph operator. Simić \cite{Simic} studied graphs whose line graphs are equal to their $k$-distance graphs. Foucaud et al$.$ \cite{FOUCAUD} discovered a strong relationship between the vertex degree of a graph and the clique number of its $k$-distance graph. Zelinka \cite{BohdanZelinka2000} studied the periodicity (see Definition \ref{def:periodicity}) of graph operators, in particular the $k$-distance operator, and showed that a graph $G$ having diameter $2$ is complementary to its $2$-distance graph if the complement of $G$ also has diameter $2$. Azimi et al$.$ \cite{Azimi2017} studied graphs that are their own $2$-distance graphs. As an obvious example, the $2$-distance graph of any odd cycle is isomorphic to the original graph.

Al-saadi and Natal \cite{alsaadi2025diameter} showed that the diameter of $2$-distance graphs is bounded. Bai et al$.$ \cite{BAI2024113493} showed that $2$-distance graphs can be recognized in polynomial-time and studied their complements. Heuvel et al$.$ \cite{VANDENHEUVEL2019143} investigated the chromatic number of $k$-distance graphs. Khormali \cite{KhormaliConnective} proved for any $k$ that if $k > 3$ and a largest cycle in a graph has $3$ vertices then its $k$-distance graph is disconnected.

In this paper, we explore the periods of the $k$-distance family of graphs. Intuitively, the $k$-distance operator can be applied to a connected graph to give a disconnected graph, but the reverse is untrue. Therefore, every graph in a period has the same number of connected components. These properties have been studied before. For example, Jafari and Musawi \cite{Jafari} classified all graphs whose $2$-distance graphs are connected. In order to capture connectedness and the period of the $k$-distance operator, we define the \emph{periodicity number} and the \emph{weak periodicity number} (Definition \ref{def:periodicity}) of a graph. Our results are organized as follows.

In Sect. \ref{sect:Cliques}, we provide a sharp upper bound on the clique number of a connected $2$-distance graph. 

Then, in Sect. \ref{sect:weak}, we show that there exists a weak period of any size under the $k$-distance operator for all $k$ (Theorem \ref{thm:omniWeak}).

Lastly, Sect. \ref{sect:strong} explores strong periodicity. Our main result is that we prove that there exists a strong period of any size under the $k$-distance operator for all $k > 2$ (Theorem \ref{thm:StrongAny}). Along the way, we provide evidence (Conjecture \ref{con:period}) that the maximum periodicity of a $2$-distance graph is $|V| - 3$.

\subsection*{Preliminaries}

Let $G = (V,E)$ be an undirected graph. By $d_G(u,v)$ we denote \emph{geodesic distance}, the length of a shortest path, between the vertices $u$ and $v$ in $G$. $\omega(G)$ denotes the \emph{clique number} of $G$, the size of a largest clique in $G$. Moreover, $\alpha(G)$ denotes the \emph{independence number} of $G$, the size of a largest independent set in $G$. For any $x \in V$, let $N_G(x) = \{y : \{x,y\} \in E\}$ be the neighborhood of $x$. We say $x \in N_G(S)$ if $x$ is adjacent to some vertex in $S$ while $x \not\in S$. We write $x \in N_G[S]$ if $x \in N_G(S) \cup S$. We denote by $C_{N_G}(S)$ the \emph{common neighborhood} of $S$ in $G$. i.e. $C_{N_G}(S) = \bigcap_{v \in S} N_G(v)$. We use $\gcd(a,b) = c$ to express that the \emph{greatest common divisor} between $a$ and $b$ is $c$. Two vertices $u$ and $v$ are \emph{false twins} if $N_G(u) = N_G(v)$ and $u \not\in N_G(v)$. Two vertices $u$ and $v$ are \emph{true twins} if $N_G(u) = N_G(v)$ and $u \in N_G(v)$. 

Let $k > 1$. The $k$-distance operator $D^k$ takes as input a simple graph $G$ and outputs its $k$-distance graph, a graph having the same vertex set $V$ where any pair $u,v$ is adjacent if and only if $d_G(u,v) = k$. Thus, $D^k(G)$ is the $k$-distance graph of $G$. A graph $Y$ is \emph{$k$-distance} if there exists a graph $X$ such that $Y \cong D^k(X)$ (meaning that $Y$ is isomorphic to $D^k(X)$). When specificity is necessary, we write $G_t^k$ to denote the $k$-distance graph of $G$ given by $t - 1$ applications of the $D^k$ operator (it is implied that $G = G_1^k$).

We contrast the standard definition of isomorphism with \emph{labeled isomorphism}. We write $G \equiv H$ if the graphs $G$ and $H$ are isomorphic and the labels of their vertices are preserved.

\begin{definition}
Let $\mathcal{G} = \langle G =  G_1^k,G_2^k,G_3^k,...\rangle$ be a sequence of graphs. Note that $G_i^k \cong D^k(G_{i-1}^k)$ for each $i>1$. Let $f > 1$ and let $G_f^k$ be the first graph in $\mathcal{G}$ such that $G_f^k \cong G_1^k$ $($\emph{resp.} $G_f^k \equiv G_1^k)$, if it exists. We call the sequence $\langle G, \dots, G_{f-1}^k \rangle$ a \emph{(strong) period} \emph{(resp.} \emph{weak period)} and any graph in the sequence is \emph{(strong) periodic} \emph{(resp.} \emph{weak periodic)}.

We say $f-1$ is equal to $\rho^k(G)$, the \emph{periodicity number}
of $G$. If no such $f$ exists, then $\rho^k(G)=0$. Similarly, if
$G_f^k$ is the first graph such that $G_f^k \equiv G_1^k$, then we say
$f-1$ is equal to $\varrho^k(G)$, the \emph{weak periodicity number}
of $G$. If no such $f$ exists, then $\varrho^k(G)=0$.
\label{def:periodicity}
\end{definition}

Under Definition \ref{def:periodicity}, it is always true that $\rho^k(G) \leq \varrho^k(G)$. Later (Proposition \ref{prop:periodIFF}), we formally prove that a graph that belongs to a strong period also belongs to a weak period. 

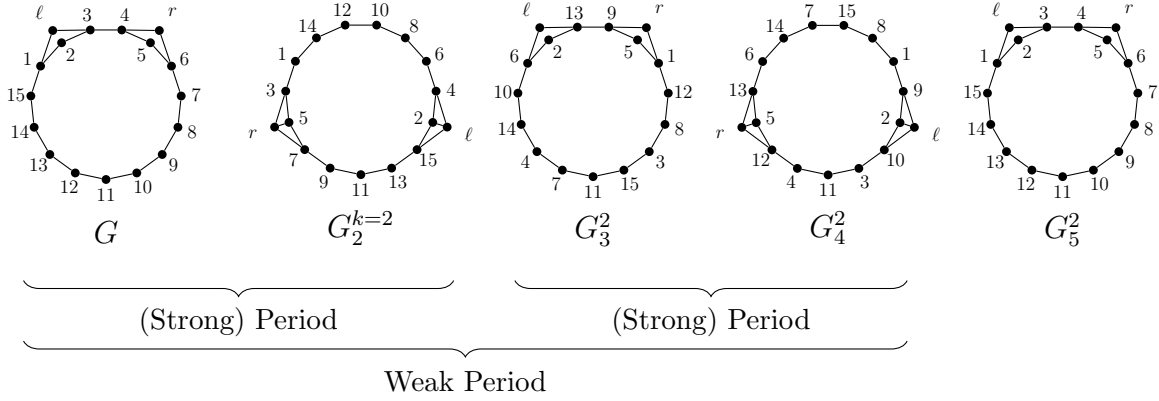
\begin{figure}[h]
    \centering
    \begin{minipage}{0.18\textwidth}
    \centering
    \begin{tikzpicture}
        \foreach \n in {1,...,15}{
            \node[circle,fill=black,scale=.3] at ({(15-\n)*360/15+174}:1cm) (n\n) {};
        }
        \foreach \n in {2,...,15}{
            \draw (n\n) -- (n\the\numexpr\n-1\relax);
        }
        \foreach \n in {1,3,4,6,7,8,9,10,11,12,13,14,15}{
            \node[scale=.6] at ({(15-\n)*360/15+174}:1.2cm) { $\n$};
        }
        \node[scale=.6] at (126:.8cm) {$2$};
        \node[scale=.6] at (54:.8cm) {$5$};

        \node[scale=.6] at (126:1.5cm) {$\ell$};
        \node[scale=.6] at (54:1.5cm) {$r$};
        
        \node[circle,fill=black,scale=.3] at (126:1.2cm) (nleft) {};
        \node[circle,fill=black,scale=.3] at (54:1.2cm) (nright) {};
        
        \draw (nleft) -- (n1);
        \draw (nleft) -- (n3);
        \draw (nright) -- (n4);
        \draw (nright) -- (n6);
        \draw (n1) -- (n15);
        \node at (0,-1.7) {$G$};
    \end{tikzpicture}
    \end{minipage}
    \hspace{.01\textwidth}
    \begin{minipage}{0.18\textwidth}
    \centering
    \begin{tikzpicture}
        \foreach \n in {1,...,15}{
            \node[circle,fill=black,scale=.3] at ({(15-\n)*360/15+174}:1cm) (n\n) {};
        }
        \foreach \n in {2,...,15}{
            \draw (n\n) -- (n\the\numexpr\n-1\relax);
        }
        \node[scale=.6] at (150:1.2cm) {$1$};
        \node[scale=.6] at (126:1.2cm) {$14$};
        \node[scale=.6] at (102:1.2cm) {$12$};
        \node[scale=.6] at (78:1.2cm) {$10$};
        \node[scale=.6] at (54:1.2cm) {$8$};
        \node[scale=.6] at (30:1.2cm) {$6$};
        \node[scale=.6] at (6:1.2cm) {$4$};
        \node[scale=.6] at (342:.8cm) {$2$};
        \node[scale=.6] at (318:1.2cm) {$15$};
        \node[scale=.6] at (294:1.2cm) {$13$};
        \node[scale=.6] at (270:1.2cm) {$11$};
        \node[scale=.6] at (246:1.2cm) {$9$};
        \node[scale=.6] at (222:1.2cm) {$7$};
        \node[scale=.6] at (198:.8cm) {$5$};
        \node[scale=.6] at (174:1.2cm) {$3$};

        \node[scale=.6] at (342:1.5cm) {$\ell$};
        \node[scale=.6] at (198:1.5cm) {$r$};
        
        \node[circle,fill=black,scale=.3] at (342:1.2cm) (nleft) {};
        \node[circle,fill=black,scale=.3] at (198:1.2cm) (nright) {};
        
        \draw (nleft) -- (n7);
        \draw (nleft) -- (n8);
        \draw (nleft) -- (n9);
        \draw (nright) -- (n13);
        \draw (nright) -- (n14);
        \draw (nright) -- (n15);
        \draw (n1) -- (n15);
        \node at (0,-1.7) {$G_2^{k=2}$};
    \end{tikzpicture}
    \end{minipage}
    \hspace{0.01\textwidth}
    \begin{minipage}{0.18\textwidth}
    \centering
    \begin{tikzpicture}
        \foreach \n in {1,...,15}{
            \node[circle,fill=black,scale=.3] at ({(15-\n)*360/15+174}:1cm) (n\n) {};
        }
        \foreach \n in {2,...,15}{
            \draw (n\n) -- (n\the\numexpr\n-1\relax);
        }
        \node[scale=.6] at (150:1.2cm) {$6$};
        \node[scale=.6] at (126:.8cm) {$2$};
        \node[scale=.6] at (102:1.2cm) {$13$};
        \node[scale=.6] at (78:1.2cm) {$9$};
        \node[scale=.6] at (54:.8cm) {$5$};
        \node[scale=.6] at (30:1.2cm) {$1$};
        \node[scale=.6] at (6:1.2cm) {$12$};
        \node[scale=.6] at (342:1.2cm) {$8$};
        \node[scale=.6] at (318:1.2cm) {$3$};
        \node[scale=.6] at (294:1.2cm) {$15$};
        \node[scale=.6] at (270:1.2cm) {$11$};
        \node[scale=.6] at (246:1.2cm) {$7$};
        \node[scale=.6] at (222:1.2cm) {$4$};
        \node[scale=.6] at (198:1.2cm) {$14$};
        \node[scale=.6] at (174:1.2cm) {$10$};

        \node[scale=.6] at (126:1.5cm) {$\ell$};
        \node[scale=.6] at (54:1.5cm) {$r$};
        
        \node[circle,fill=black,scale=.3] at (126:1.2cm) (nleft) {};
        \node[circle,fill=black,scale=.3] at (54:1.2cm) (nright) {};
        
        \draw (nleft) -- (n1);
        \draw (nleft) -- (n3);
        \draw (nright) -- (n4);
        \draw (nright) -- (n6);
        \draw (n1) -- (n15);
        \node at (0,-1.7) {$G_3^2$};
    \end{tikzpicture}
    \end{minipage}
    \begin{minipage}{0.18\textwidth}
    \centering
    \begin{tikzpicture}
        \foreach \n in {1,...,15}{
            \node[circle,fill=black,scale=.3] at ({(15-\n)*360/15+174}:1cm) (n\n) {};
        }
        \foreach \n in {2,...,15}{
            \draw (n\n) -- (n\the\numexpr\n-1\relax);
        }
       \node[scale=.6] at (150:1.2cm) {$6$};
        \node[scale=.6] at (126:1.2cm) {$14$};
        \node[scale=.6] at (102:1.2cm) {$7$};
        \node[scale=.6] at (78:1.2cm) {$15$};
        \node[scale=.6] at (54:1.2cm) {$8$};
        \node[scale=.6] at (30:1.2cm) {$1$};
        \node[scale=.6] at (6:1.2cm) {$9$};
        \node[scale=.6] at (342:.8cm) {$2$};
        \node[scale=.6] at (318:1.2cm) {$10$};
        \node[scale=.6] at (294:1.2cm) {$3$};
        \node[scale=.6] at (270:1.2cm) {$11$};
        \node[scale=.6] at (246:1.2cm) {$4$};
        \node[scale=.6] at (222:1.2cm) {$12$};
        \node[scale=.6] at (198:.8cm) {$5$};
        \node[scale=.6] at (174:1.2cm) {$13$};

        \node[scale=.6] at (342:1.5cm) {$\ell$};
        \node[scale=.6] at (198:1.5cm) {$r$};
        
        \node[circle,fill=black,scale=.3] at (342:1.2cm) (nleft) {};
        \node[circle,fill=black,scale=.3] at (198:1.2cm) (nright) {};
        
        \draw (nleft) -- (n7);
        \draw (nleft) -- (n8);
        \draw (nleft) -- (n9);
        \draw (nright) -- (n13);
        \draw (nright) -- (n14);
        \draw (nright) -- (n15);
        \draw (n1) -- (n15);
        \node at (0,-1.7) {$G_4^2$};
    \end{tikzpicture}
    \end{minipage}
    \hspace{0.01\textwidth}
    \begin{minipage}{0.18\textwidth}
    \centering
    \begin{tikzpicture}
        \foreach \n in {1,...,15}{
            \node[circle,fill=black,scale=.3] at ({(15-\n)*360/15+174}:1cm) (n\n) {};
        }
        \foreach \n in {2,...,15}{
            \draw (n\n) -- (n\the\numexpr\n-1\relax);
        }
        \foreach \n in {1,3,4,6,7,8,9,10,11,12,13,14,15}{
            \node[scale=.6] at ({(15-\n)*360/15+174}:1.2cm) { $\n$};
        }
        \node[scale=.6] at (126:.8cm) {$2$};
        \node[scale=.6] at (54:.8cm) {$5$};

        \node[scale=.6] at (126:1.5cm) {$\ell$};
        \node[scale=.6] at (54:1.5cm) {$r$};
        
        \node[circle,fill=black,scale=.3] at (126:1.2cm) (nleft) {};
        \node[circle,fill=black,scale=.3] at (54:1.2cm) (nright) {};
        
        \draw (nleft) -- (n1);
        \draw (nleft) -- (n3);
        \draw (nright) -- (n4);
        \draw (nright) -- (n6);
        \draw (n1) -- (n15);
        \node at (0,-1.7) {$G_5^2$};
    \end{tikzpicture}
    \end{minipage}

    \begin{tikzpicture}
        \node (G1) at (-7.5,0) {};
        \node (G2) at (-2.1,0) {};
        \node (G3) at (-1,0) {};
        \node (G4) at (3.9,0) {};
        \node (G5) at (7,0) {};

        \draw[decorate,
           decoration={brace,mirror,amplitude=6pt},
            yshift=-8pt]
            (G1.south west) -- (G2.south east)
            node[midway, below=6pt] {\small (Strong) Period};
        \draw[decorate,
           decoration={brace,mirror,amplitude=6pt},
            yshift=-8pt]
            (G3.south west) -- (G4.south east)
            node[midway, below=6pt] {\small (Strong) Period};

        \draw[decorate,
            decoration={brace,mirror,amplitude=6pt,raise=0.8cm}]
            (G1.south west) -- (G4.south east)
            node[midway, below=1.05cm] {\small Weak Period};
    \end{tikzpicture}
    
    \caption{In the figure above, we show an example of $\mathcal{G}$ as described in Definition \ref{def:periodicity}. Observe that $\langle G, G_2^{k=2} \rangle$ or $\langle G_3^2, G_4^2 \rangle$ are periods and therefore $\rho^2(G) = 2$. The first four graphs make up a weak period (note that $G \equiv G_5^2$) and thus $\varrho^2(G) = 4$.}
    \label{fig:weakVsStrong}
\end{figure}

\section{Cliques in 2-distance Graphs}
\label{sect:Cliques}

In this section, we explore cliques in $2$-distance graphs. First, we show that clique number is upper bounded by the number of vertices minus $3$ in connected $2$-distance graphs, and that this inequality is sharp. Then, we characterize the existence of triangles (cliques of size $3$) in $2$-distance graphs.

\begin{proposition}
If $G = (V,E)$ is a connected $2$-distance graph, then $\omega(G) \leq |V| - 3$.
\label{prop:cliqueNum}
\end{proposition}

\begin{proof}
Suppose not. Let $H$ be a graph such that $D^2(H) = G$. Moreover, let $A$ be a largest clique in $G$. Clearly, $\omega(G)$ cannot have more vertices than $|V|$. Therefore, we consider three cases based on the value of $\omega(G)$.

\textbf{Case 1.} $\omega(G) = |V|$. In this case, $G$ is clearly a complete graph. Note that any edge in $G$ does not exist in $H$. Given that $G$ is a clique, $H$ contains no edges. Thus, $H$ is an independent set. Since $H$ is disconnected, $G$ cannot be connected, a contradiction.

\textbf{Case 2.} $\omega(G) = |V| - 1$. There exists one vertex $d \not\in A$. Since $G$ is connected, we have that $d \in N_G(v)$ where $v \in A$. Note that $d_H(d,v) = 2$, meaning that there exists some $w \in C_{N_H}(d,v)$. However, $v$ neighbors every vertex in $G$ (including $w$), so $w$ cannot exist, a contradiction.

\textbf{Case 3.} $\omega(G) = |V| - 2$. There are exactly two vertices not in $A$. Let $d,b \not\in A$. Since $G$ is connected, at least one of ${d,b}$ must be adjacent to some vertex in $A$. W.l.o.g., let $d$ be adjacent to $a_1 \in A$ in $G$. Thus, there exists a vertex $w \in C_{N_H}(d,a_1)$. If $w \in A$, then $a_1 \not\in N_G(w)$, a contradiction to the assumption that $A$ is a clique in $G$. Therefore, $w \not\in A$. It must be that $w = b$. Observe that $b \in C_{N_H}(d,a_1)$ and hence $b \not\in N_G(d)$ and $b \not\in N_G(a_1)$. So, there exists $a_2 \in A \setminus \{a_1\}$ where $b \in N_G(a_2)$. There also exists $v \in C_{N_H}(b,a_2)$. Of course, $v \not\in A$ in $G$ because $A$ is a clique that contains $a_2$. It must be that $d = v$. We have shown that $b \not\in N_H(a_2)$ and $d \not\in N_H(a_1)$.

Notice that $C_{N_H}(a_1,a_2) \neq \emptyset$. Let $u \in C_{N_H}(a_1,a_2)$. Of course, $u \neq b$ and $u \neq d$ because $b \not\in N_H(a_2)$ and $d \not\in N_H(a_1)$. We have that $u \in A$. But, recall that $u \in C_{N_H}(a_1,a_2)$ implies that $u \not\in N_G(a_1,a_2)$, a contradiction to the fact that $A$ is a clique in $G$.
\end{proof}

Fig. \ref{fig:cliqueStrong} shows that Proposition \ref{prop:cliqueNum} is sharp.

\begin{figure} [h]
    \centering
    \begin{tikzpicture}[scale=1]
        \node[circle,fill=ForestGreen,scale=.7] at (0:0) (center) {};
        \node[circle,fill=blue,scale=.7] at (-1cm,1.5cm) (top) {};
        \node[circle,fill=red,scale=.7] at (1cm,1.5cm) (right) {};
        \node[circle,fill=lightgray,scale=.7] at ([shift={(1.5cm,0)}]right.west) (far) {};
        \node[circle,fill=black,scale=.7] at (-1.5cm,-1.5cm) (bottom1) {};
        \node[circle,fill=black,scale=.7] at (-1cm,-1.5cm) (bottom2) {};
        \node[circle,fill=black,scale=.7] at (-.5cm,-1.5cm) (bottom3) {};
        \node[circle,fill=black,scale=.15] at (0,-1.5cm) (dot1) {};
        \node[circle,fill=black,scale=.15] at (.25cm,-1.5cm) (dot2) {};
        \node[circle,fill=black,scale=.15] at (.5cm,-1.5cm) (dot3) {};
        \node[circle,fill=black,scale=.7] at (1cm,-1.5cm) (bottom4) {};

        \node at (-1cm,1.9cm) {$b$};
        \node at (-1.5cm,-1.9cm) {$a_1$};
        \node at (-1cm,-1.9cm) {$a_2$};
        \node at (-.5cm,-1.9cm) {$a_3$};
        \node at (1cm,-1.9cm) {$a_p$};

        \node at (1cm,-2.36cm) { };

        \draw (center) -- (top);
        \draw (center) -- (right);
        \draw (top) -- (right);
        \draw (right) -- (far);
        \draw (bottom1) -- (center);
        \draw (bottom3) -- (center);
        \draw (center) -- (bottom2);
        \draw (center) -- (bottom4);
    \end{tikzpicture}
    ~~~~~~~~
    \begin{tikzpicture}[scale=1]
        
        \node[circle,fill=red,scale=.7] at (-1.3cm,1.5cm) (top) {};
        \node[circle,fill=blue,scale=.7] at (.7cm,1.5cm) (right) {};
        \node[circle,fill=lightgray,scale=.7] at ([shift={(1.5cm,0)}]right.west) (far) {};
        
        \node[circle,fill=ForestGreen,scale=.7] at ([shift={(1.5cm,0)}]far.west) (bleft) {};
        
        \node[circle,fill=black,scale=.15] at (0,-1.5cm) (dot1) {};
        \node[circle,fill=black,scale=.15] at (.25cm,-1.5cm) (dot2) {};
        \node[circle,fill=black,scale=.15] at (.5cm,-1.5cm) (dot3) {};

        \node[circle,fill=black,scale=.7] at (-1.5cm,-1.5cm) (bottom1) {};
        \node[circle,fill=black,scale=.7] at (-1cm,-1.5cm) (bottom2) {};
        \node[circle,fill=black,scale=.7] at (-.5cm,-1.5cm) (bottom3) {};
        \node[circle,fill=black,scale=.7] at (1cm,-1.5cm) (bottom4) {};

        \node at (0.7cm,1.9cm) {$b$};
        \node at (-1.6cm,-1.9cm) {$a_1$};
        \node at (1.1cm,-1.9cm) {$a_p$};
        
        \draw (top) -- (bottom1);
        \draw (top) -- (bottom2);
        \draw (top) -- (bottom3);
        \draw (top) -- (bottom4);

        \draw (right) -- (bottom1);
        \draw (right) -- (bottom2);
        \draw (right) -- (bottom3);
        \draw (right) -- (bottom4);

        \draw (bottom1) -- (bottom2);
        \draw (bottom1) to[out=-60, in=-120] (bottom3);
        \draw (bottom1) to[out=-70, in=-110] (bottom4);

        \draw (bottom2) -- (bottom3);
        \draw (bottom2) to[out=-60, in=-120] (bottom4);

        \draw (bottom3) to[out=-45, in=-135] (bottom4);

        \draw (right) -- (far);
        \draw (far) -- (bleft);
    \end{tikzpicture}

    \caption{The family of graphs on the right are the $2$-distance graphs of the graphs on the left. On the right, the set $\{a_1,a_2,\dots,a_p,b\}$ forms a clique satisfying $\omega(G) = |V| - 3$. Colors are used to label vertices to provide clarity.}
    \label{fig:cliqueStrong}
\end{figure}
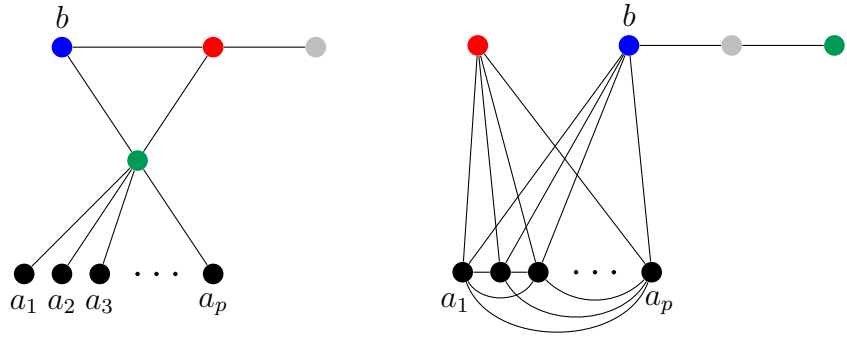

\begin{proposition} Let $G$ be a graph. Then, $G_2^{k=2}$ contains a triangle if and only if $G$ contains either an induced claw or a $C_6$ where $\alpha(C_6) = 3$.
\end{proposition}

\begin{proof}

(Necessity) $\Leftarrow$ First, suppose that $G$ has an induced claw $W = \langle t,u,v,w \rangle$ where $t$ is the center vertex. Trivially, the distance between each pair of vertices in $W \setminus \{t\}$ is $2$. This implies that there exists a triangle in $G_2^{k=2}$.
    
Instead, suppose that there exists a $C_6$ having $\alpha(C_6) = 3$ in $G$. Let $T = \langle u,v,w \rangle$ form a stable set in $C_6$. It follows that each pair in $T$ has a unique common neighbor in $G$. This implies that $\{u,v,w\}$ is a triangle in $G_2^{k=2}$, so we are done.

(Sufficiency) $\Rightarrow$ Suppose not and let $T = \langle u,v,w \rangle$ be vertices of a triangle in $G_2^{k=2}$. Certainly, all pairs $a,b \in T$ have $d_G(a,b) = 2$ and $C_{N_G}(a,b) \neq \emptyset$. There are two cases to consider based on the number of unique common neighbors between vertices in $T$.

\textbf{Case 1.} $|C_{N_G}(T)| \geq 1$. Let $x \in C_{N_G}(T)$. Now, $T \cup \{x\}$ is a claw and by assumption it is not induced. However, now for some pair $a,b \in T$ we have $a \in N_G(b)$, a contradiction to the fact that $d_G(a,b) = 2$.
    
\textbf{Case 2.} $|C_{N_G}(T)| = 0$. Consequently, each pair in $T$ has a unique common neighbor. So, $T$ and its common neighbors form a $C_6 = \{u,a,v,b,w,c\}$. If $\alpha(C_6) \neq 3$, then some pair in $T$ is adjacent, a contradiction.
\end{proof}

\section{Weak Periodicity}
\label{sect:weak}

We will show that a weak period of every size exists under the $k$-distance operator for all $k$. In particular, there exists a cycle graph $H$ such that $\varrho^k(H) = p$ for any $p > 1$. This result (Theorem \ref{thm:omniWeak}) builds on a series of lemmas which we will now begin.

\begin{lemma}
Let $G = \langle v_0, v_1, \dots , v_{|V|-1} \rangle$ be a cycle graph. We have that if $|V| > 2k$, then $v_0 \in N_{G_2^k}(v_k)$.
\label{lem:kadj}
\end{lemma}

\begin{proof}
Clearly, there exist two $v_0,v_k$-paths. One $v_0,v_k$-path $P_1$ clearly has length $k$ and contains $v_1$. Noting that $G$ is a cycle, and thus $v_{|V|-1} \in N_G(v_0)$, the other $v_0,v_k$-path $P_2$ contains $v_{|V|-1}$ and has length $|V| - k$. 

The lengths of $P_1$ and $P_2$, combined with the assumption that $|V| > 2k$, give us that $P_1$ is a shortest path. This immediately implies that $v_0 \in N_{G_2^k}(v_k)$, so we are done.
\end{proof}

Later, we will show (Corollary \ref{lem:modKP}) that Lemma \ref{lem:kadj} can be applied repeatedly while preserving the structure of the original graph $G$ when $G$ is a sufficiently large cycle and $\gcd(k,|V|) = 1$. To achieve this result, we provide the following, intermediate lemma.

\begin{lemma}
Let $G = \langle v_0, v_1, \dots , v_{|V|-1} \rangle$ be a cycle graph having $|V| > 2k$. We have that $G_2^k$ is a cycle if and only if $\gcd(k,|V|) = 1$.
\label{lem:gcd}
\end{lemma}

\begin{proof}
(Necessity) $\Leftarrow$ Suppose for the sake of contradiction that $\gcd(k,|V|) = 1$ but that $G_2^k$ is not a cycle. Given that $|V| > 2k$, it is easy to see that every vertex in $G_2^k$ has degree $2$. Thus, it must be that $G_2^k$ is either isomorphic to $G$ or the disjoint union of smaller cycles. Since $G_2^k$ is not a cycle by our assumption, we have that $G_2^k$ is formed of disjoint cycles. Let $G_2^k$ have $\ell$ components $\{C_1,C_2, \dots C_{\ell}\}$. By Lemma \ref{lem:kadj} where $v_0 := v_i$, each $C_i$ has $v_{i} \in N_{G_2^k}(v_{(k+i)\mod{|V|}})$. It is easy to see that each $C_i$ has an equal number of vertices. In particular, the vertices $A = \{v_0,\dots,v_{k-1}\} \subseteq G$ must be divided equally among the $\ell$ components in order for each $C_i$ to exist. In other words, $k \mod \ell = 0$. See Fig. \ref{fig:modifiedModKP}, which shows an example where $k = 3$ and $\ell = 3$.

Let $n$ be the number of vertices in any $C_i$. Obviously, $|V| = n \ell$. Since $|V|$ and $k$ are divisible by $\ell$, we have that $\gcd(k,|V|) \geq \ell$. If $\ell = 1$, then $G_2^k$ has one connected component, a contradiction. Thus, $\ell \neq 1$, an immediate contradiction to our assumption that $\gcd(k,|V|) = 1$.

(Sufficiency) $\Rightarrow$ Suppose not. Thus, $G_2^k$ is a connected cycle but that $\gcd(k,|V|)$ $ \neq 1$. For all $d > 0$, note that $v_0$ and $v_{dk}$ belong to the only cycle in $G_2^k$. So, $v_1 = v_{dk \mod{|V|}}$ for some $d$. This means that $1 = (dk \mod{|V|})$, or $dk = w|V| + 1$ for some $w$, further simplified to $c = dk - w|V|$. It is trivial that $c = 1$. 

Let $\gcd(k,|V|) = \ell$, where $\ell \neq 1$. Since $k$ and $|V|$ are divisible by $\ell$, we have that $\ell$ is a common factor of both terms in $c$ and therefore $c$ is divisible by $\ell$. However, the fact that $c = 1$ implies that $1$ is divisible by $\ell$. This is only possible when $\ell = gcd(k, |V|) = 1$, a contradiction.  
\end{proof}

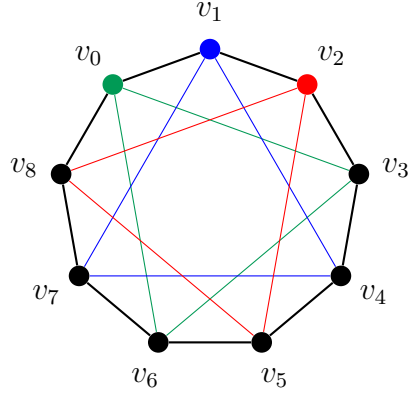
\begin{figure}[h]
    \centering
    \begin{tikzpicture}[scale=1]
        \node[circle,fill=ForestGreen,scale=.7] at ({9*360/9+130}:2cm) (n0) {};
        \foreach \n in {3,6}{
            \node[circle,fill=black,scale=.7] at ({(9-\n)*360/9+130}:2cm) (n\n) {};
        }
        \node[circle,fill=blue,scale=.7] at ({(9-1)*360/9+130}:2cm) (n1) {};
        \foreach \n in {4,7}{
            \node[circle,fill=black,scale=.7] at ({(9-\n)*360/9+130}:2cm) (n\n) {};
        }
        \node[circle,fill=red,scale=.7] at ({(9-2)*360/9+130}:2cm) (n2) {};
        \foreach \n in {5,8}{
            \node[circle,fill=black,scale=.7] at ({(9-\n)*360/9+130}:2cm) (n\n) {};
        }
        \foreach \n in {0,...,7}{
            \node at ({(9-\n)*360/9+130}:2.5cm) {$v_{\n}$};
            \draw [thick](n\n) -- (n\the\numexpr\n+1\relax);
        }
        \draw [thick](n0) -- (n8);
        \node at ({(9-8)*360/9+130}:2.5cm) {$v_{8}$};
        \draw [ForestGreen] (n0) -- (n6);
        \draw [ForestGreen] (n6) -- (n3);
        \draw [ForestGreen] (n3) -- (n0);
        \draw [blue] (n1) -- (n7);
        \draw [blue] (n7) -- (n4);
        \draw [blue] (n4) -- (n1);
        \draw [red] (n2) -- (n8);
        \draw [red] (n8) -- (n5);
        \draw [red] (n5) -- (n2);
    \end{tikzpicture}

    \caption{Let the black edges of the figure induce graph $G$. Then, let the colored edges induce $G_2^{k=3} = H$. Note that $H$ is the disjoint union of three triangles. The colored vertices $v_0$, $v_1$, and $v_2$ in $G$ are equally partitioned between the components of $H$.}
    \label{fig:modifiedModKP}
\end{figure}

\begin{corollary}
Let $p \geq 0$ and $G = \langle v_0, v_1, \dots , v_{|V|-1} \rangle$ be a cycle graph having $|V| > 2k$ and $\gcd(k,|V|)=1$. Then, $G^k_{p+1}$ is a cycle and $v_0 \in N_{G_{p+1}^k}(v_{k^p \mod{|V|}})$.
\label{lem:modKP}
\end{corollary}

\begin{proof}
The statement is certainly true if $p = 0$ because $v_0 \in N_G(v_1)$ and $G$ is assumed to be a cycle (noting that $G = G^k_{1}$).

We claim the statement is true for any $p = i$ where $i > 0$. First, we show the statement is true when $i = 1$. Since $k < |V|$ and by Lemma \ref{lem:kadj} we have that $v_0 \in N_{G^k_2}(v_k)$. Of course, the fact that $\gcd(|V|,k)=1$ and $|V(G)| = |V(G_{2}^k)|$ implies by Lemma \ref{lem:gcd} that $G^k_2$ is a cycle.

Now suppose that $i > 1$. By inductive hypothesis, $v_0 \in N_{G_{i+1}^k}$ $(v_{k^i \mod{|V|}})$ and $G_{i+1}^k$ is a cycle. By Lemma \ref{lem:kadj} where $G := G_{i+1}^k$ we have that $v_0 \in N_{G_{i+2}^k}$ $(v_{k^{i+1} \mod |V|})$.  Of course, the fact that $\gcd(|V(G_{i+1}^k)|,k) = 1$ and $|V(G_{i+1}^k)| = |V(G_{i+2}^k)|$ implies by Lemma \ref{lem:gcd} that $G_{i+2}^k$ is a cycle.
\end{proof}

We are prepared to complete our main result.

\begin{theorem}
Let $k > 1$ and $p > 1$. The cycle graph $C_{k^p+1}$ has weak-periodicity $p$ under $k$-distance.
\label{thm:omniWeak}
\end{theorem}

\begin{proof}
Formally, we will prove there exists an $H$ such that $\varrho^k(H) = p$. Let cycle graph $G = (V,E)$ have $|V| = k^p + 1$.

First, we will show that all cycles with $|V| = k^p+1$ satisfy the assumptions of Lemma \ref{lem:gcd}. Clearly, we see that $2k < (k^p + 1)$. Moreover, we have that $\gcd(k,|V|) = 1$ because there are no common divisors between $k$ and $k^p+1$ other than $1$. In particular, $G_{p+1}^k$ is a connected cycle. By Corollary \ref{lem:modKP} we have $v_0 \in N_{G_{p+1}^k}(v_{k^p \mod |V|})$. This is equivalent to saying that $v_0 \in N_{G_{p+1}^k}(v_{k^p})$ because $|V| > k^p$. 

We claim that $G_1^k \equiv G_{p+1}^k$. For each $0 < i \leq |V|$, by Corollary \ref{lem:modKP} where $G := \{v_i v_{i+1} \dots  $ $v_{|V|-1} v_0 v_1  \dots$ $v_{i-1}\}$, $k := k$, and $p := p$, we have $v_i \in N_{G_{p+1}^k}(v_{(i+k^{p})\mod |V|})$. By construction, we have proven our claim.

Now we will prove that $G_{p+1}^k$ is the first recurrence of a graph that is label isomorphic to $G_1^k$ in the sequence $\langle G_1^k, G_2^k, $ $ \dots \rangle$. For the sake of contradiction, suppose that there is a graph $G_{\ell+1}^k$ such that $G_{\ell+1}^k \equiv G_1^k$ and $0 < \ell < p$. By definition, $v_0 \in N_{G_{\ell+1}^k}(v_1,v_{|V|-1})$. By Corollary \ref{lem:modKP} where $G:=G$, $k:=k$, and $p := \ell$, we have $v_0 \in N_{G_{\ell+1}^k}(v_{k^{\ell} \mod{|V|}})$ and therefore it is certain that $(k^\ell \mod |V|) = 1$ or $(k^\ell \mod |V|) = |V|-1$. For either of these equalities to hold, it is necessary that $k^\ell \geq |V|-1$. So, we have $k^\ell \geq k^p+1-1$ (or more simply we have $k^\ell \geq k^p$), a contradiction. We set $H := G_{p+1}^k$ and we are done.
\end{proof}

\section{Strong Periodicity}
\label{sect:strong}

In order to better study the structure of $k$-distance graphs, we consider a stronger form of a periodicity. Many authors have studied the (strong) periodicity of $k$-distance graphs. More recently, Erickson et al$.$ \cite{erickson} showed that only the disjoint union of any number of $C_7$'s and $C_9$'s have $\rho^2(G) = 3$. This result does not exclude the possibility that a $\rho^k(G)$ of any value exists for all $k > 2$. In the remaining section, we will show that this statement is true. Lastly, we will provide evidence that the periodicity number of a $2$-distance graph is upper bounded by a function of its number of vertices.

In a sense, a (strong) period is generally more difficult to find than a weak period because isomorphism must be computed in the search. Regardless, every weak period contains a (strong) period. 

\begin{proposition}
A simple graph $G$ is (strong) periodic if and only if $G$ is weak periodic.
\label{prop:periodIFF}
\end{proposition}

\begin{proof}
(Sufficiency) $\Rightarrow$ By Definition \ref{def:periodicity}, we have that, if $G$ is periodic, then an integer $u$ and a graph $G_{u+1}^k$ exist such that $G \cong G_{u+1}^k$. Moreover, it is certainly true that $G \cong G_{tu+1}^k$ for any $t > 0$. Although there are an infinite number of $G_{tu+1}^k$'s that are isomorphic to $G$ in this sequence, there are only a finite number of ways to permute vertex labels because $G$ is finite. Therefore, this subsequence must eventually repeat. Thus, $G$ is weak periodic, i.e. $\varrho^k(G)>0$, so we are done.

(Necessity) $\Leftarrow$ By Definition \ref{def:periodicity}, we have that if $G$ is weak periodic then for some $u > 0$ there exists a graph $G_{u+1}^k$ such that $G \equiv G_{u+1}^k$. Immediately we have $G \cong G_{u+1}^k$, so we are done.
\end{proof}

\begin{corollary}
For a simple, periodic graph $G$, we have that $\varrho^k(G) \mod \rho^k(G) = 0$.
\end{corollary}

\begin{proof}
Let $\mathcal{G} = \langle G =  G_1^k,G_2^k,G_3^k,...\rangle$. By Definition \ref{def:periodicity}, we have that if $G$ is periodic, then an integer $u$ and a graph $G_{u+1}^k$ exist such that $G \cong G_{u+1}^k$ and $\rho^k(G) = u$. It follows that $G \cong G_{tu+1}^k$ for any $t > 0$. For any $q$ $\neq tu+1$, it is trivial that $G \not\equiv G^k_{q}$ (in other words, $G$ may not be label isomorphic to any $G^k_i \in \mathcal G$ unless $i-1$ is a multiple of $u$). But, since $G$ also belongs to a weak period by Proposition \ref{prop:periodIFF}, for some $t' > 0$ we have that $G \equiv G_{t'u+1}^k$. Because $(t'u \mod u) = 0$, we are done.
\end{proof}

We define a class of graphs that will be pivotal to the proof of our theorem.

\begin{definition}
An \emph{ear graph} $G = (V,E)$ is a graph that contains an induced cycle of size $|V| - 2$ such that two vertices $a,b$ belonging to the cycle have false twins. Let $n = |V| - 2$ and $d = d_G(a,b)$. To be more specific, we may say that $G$ is an $(n,d)$-ear. 
\label{def:ear}
\end{definition}

Fig. \ref{fig:earGraph} shows an example of an ear graph. The next statement is true in part due to the observation that false twins remain false twins under any $k$-distance operator where $k > 2$.

\begin{figure}[h]
    \centering
    \begin{tikzpicture}[scale=.8]
        \node[circle,fill=black,scale=.7] at ({13*360/13+146}:2.5cm) (Lear) {};
        \node[circle,fill=black,scale=.7] at ({9*360/13+146}:2.5cm) (Rear) {};

        \node at (1.2cm,1.1cm) {$b$};
        \node at (-1.2cm,1.1cm) {$a$};
        \node at (2cm,1.9cm) 
        {$b'$};
        \node at (-2cm,1.9cm) {$a'$};
        
        \foreach \n in {0,1,2,3,4,5,6,7,8,9,10,11,12}{
            \node[circle,fill=black,scale=.7] at ({(13-\n)*360/13+146}:2cm) (n\n) {};
        }
        \foreach \n in {0,...,11}{
            \draw (n\n) -- (n\the\numexpr\n+1\relax);
        }
        \draw (n0) -- (n12);
        \draw (Lear) -- (n12);
        \draw (Lear) -- (n1);
        \draw (Rear) -- (n3);
        \draw (Rear) -- (n5);
             
    \end{tikzpicture}
    ~~~~~~~~~~
    \begin{tikzpicture}[scale=.9]
        \node[circle,fill=black,scale=.7] at ({12*360/13+132}:2.5cm) (Lear) {};
        \node[circle,fill=black,scale=.7] at ({11*360/13+132}:2.5cm) (Rear) {};

        \node at (0.4cm,1.5cm) {$b$};
        \node at (-0.4cm,1.45cm) {$a$};
        \node at (1cm,2.5cm) {$b'$};
        \node at (-1cm,2.5cm) {$a'$};
        
        \foreach \n in {0,1,2,3,4,5,6,7,8,9,10,11,12}{
            \node[circle,fill=black,scale=.7] at ({(13-\n)*360/13+132}:2cm) (n\n) {};
        }
        \foreach \n in {0,...,11}{
            \draw (n\n) -- (n\the\numexpr\n+1\relax);
        }
        \draw (n0) -- (n12);
        \draw (Lear) -- (n0);
        \draw (Lear) -- (n2);
        \draw (Rear) -- (n1);
        \draw (Rear) -- (n3);
        \draw (Lear) -- (Rear);
             
    \end{tikzpicture}
    
    \caption{The graph on the left is a $(13,4)$-ear graph while the graph on the right is a $(13,1)$-ear graph. $\{a,a'\}$ and $\{b,b'\}$ are false twins.}
    \label{fig:earGraph}
\end{figure}
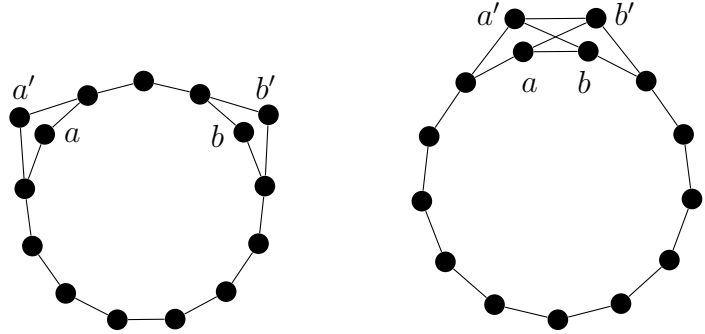

\begin{lemma}
Let $k > 2$ and let $G$ be an $(n,d)$-ear graph for any arbitrary $n,d$. Let $C_n$ be a largest induced cycle in $G$. If $\rho^k(C_n) = 1$ then $G^k_2$ is an ear graph.
\label{lem:earToEar}
\end{lemma}

\begin{proof}
Let $a',b' \in (G \setminus C_n)$ be the false twins of $a,b \in C_n$. It is quick to verify that $G_2^k$ contains a graph that is isomorphic to $C_n$ because every shortest path in $C_n$ is a shortest path in $G$. Since $N_{G}(a) = N_{G}(a')$, the set of vertices that are $k$-distant from $a$ are also $k$-distant from $a'$. Thus, $N_{G^k_2}(a) = N_{G^k_2}(a')$. In fact, any false twins (\emph{viz.} $a,a'$) remain false twins if and only if $k > 2$. (When $k = 2$, false twins become true twins). Similarly, $N_{G^k_2}(b) = N_{G^k_2}(b')$. Thus, $G_2^k$ satisfies the description of an ear graph by Definition \ref{def:ear}.
\end{proof}

\begin{figure}
    \centering
    \includegraphics[scale=.23]{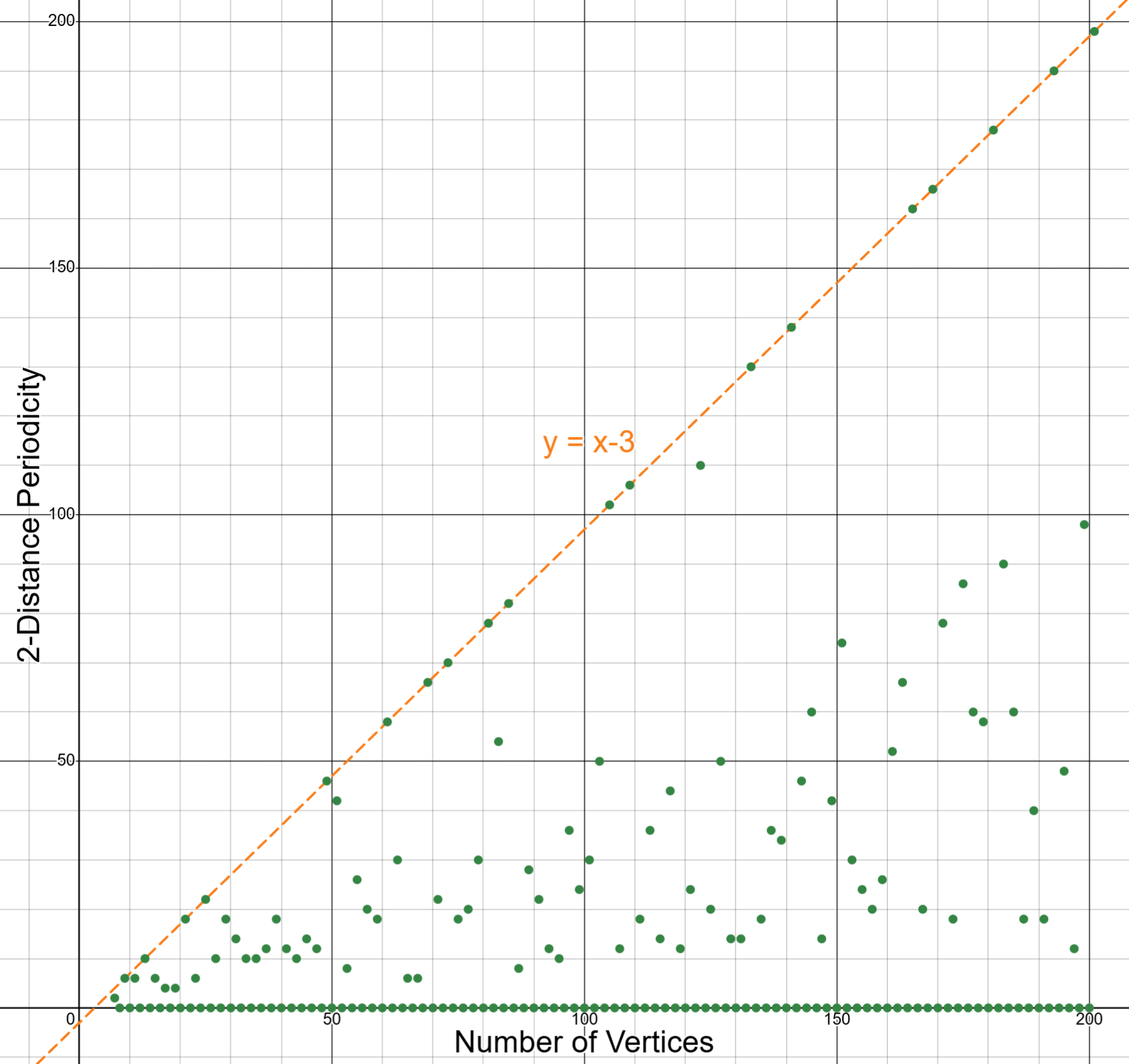}
    \caption{In the above plot, the green points represent the value of $\rho^2(G)$ where $G = (V,E)$ is an $(n,2)$-ear graph for each $n \leq 200$. Note that $\rho^2(G) \leq |V| - 3$ is sharp for only some values of $|V|$.}
    \label{fig:Desmos}
\end{figure}

Next, we will show that the class of ears in which the two cycle vertices with false twins are adjacent has useful properties.

\begin{lemma}
    Given $G = (V,E)$ and $n = |V| - 2$, let $G$ be an $(n,1)$-ear. Let $C$ be a largest induced cycle in $G$ such that $\rho^k(C) = 1$ where $k > 2$. Note that $C$ has $n$ vertices. We have that $G \cong G_q^k$ if and only if $C \equiv C_q^k$ for any $q > 1$.
    \label{lem:earStrongIsWeak}
\end{lemma}

\begin{proof}
    (Sufficiency) $\Rightarrow$ Let $C = \langle v_0,v_1,\dots,v_{n-1}\rangle$. W.l.o.g., let $v_0$ and $v_1$ have false twins $a$ and $b$, respectively, in $G$. For any $f$, we have that $G_f^k$ is an ear graph by Lemma \ref{lem:earToEar} where $G := G_{f-1}^k$, $n := n$, and $d := 1$. In addition, the fact that there are only two sets of false twins in $G$ (and that they remain false twins in any $G_f^k$) means that $a \in N_{G^k_q}(b)$ and $v_0 \in N_{G^k_q}(v_1)$.
    
    The existence of $\{a,b\} \subseteq G$ does not shorten the distances between any other pairs in $C$ because $C$ is a simple cycle. Let $0 \leq j < i < n$ and $0 \leq j',i' < n$. For every $t \leq q$ if any vertex $v_i \in N_{G_t^k}(v_j)$ then $v_{i'} \in N_{G_t^k}(v_{j'})$ provided that $i - j = i'- j'$. In particular, this is true when $t = q$. This fact, combined with $v_0 \in N_{G^k_q}(v_1)$, shows that $C \equiv C^k_q$.

    (Necessity) $\Leftarrow$ Let cycle $C = \langle v_0,v_1,\dots,v_{n-1}\rangle$ where $v_0$ and $v_1$ have false twins $a$ and $b$, respectively, in $G$. Since it is trivial that false twins remain false twins under the $k$-distance operator (where $k > 2$), we have that $N_G(a) = N_{G^k_q}(a)$ and $N_G(b) = N_{G^k_q}(b)$. This fact, combined with the assumption that $C \equiv C^k_q$ and thus $C \cong C^k_q$, gives $G \cong G^k_q$.
\end{proof}

We are prepared to complete the main result of this section.

\begin{figure}
    \centering
    \includegraphics[trim=3.25cm 5.2cm 3.25cm 3cm, clip,width=\linewidth]{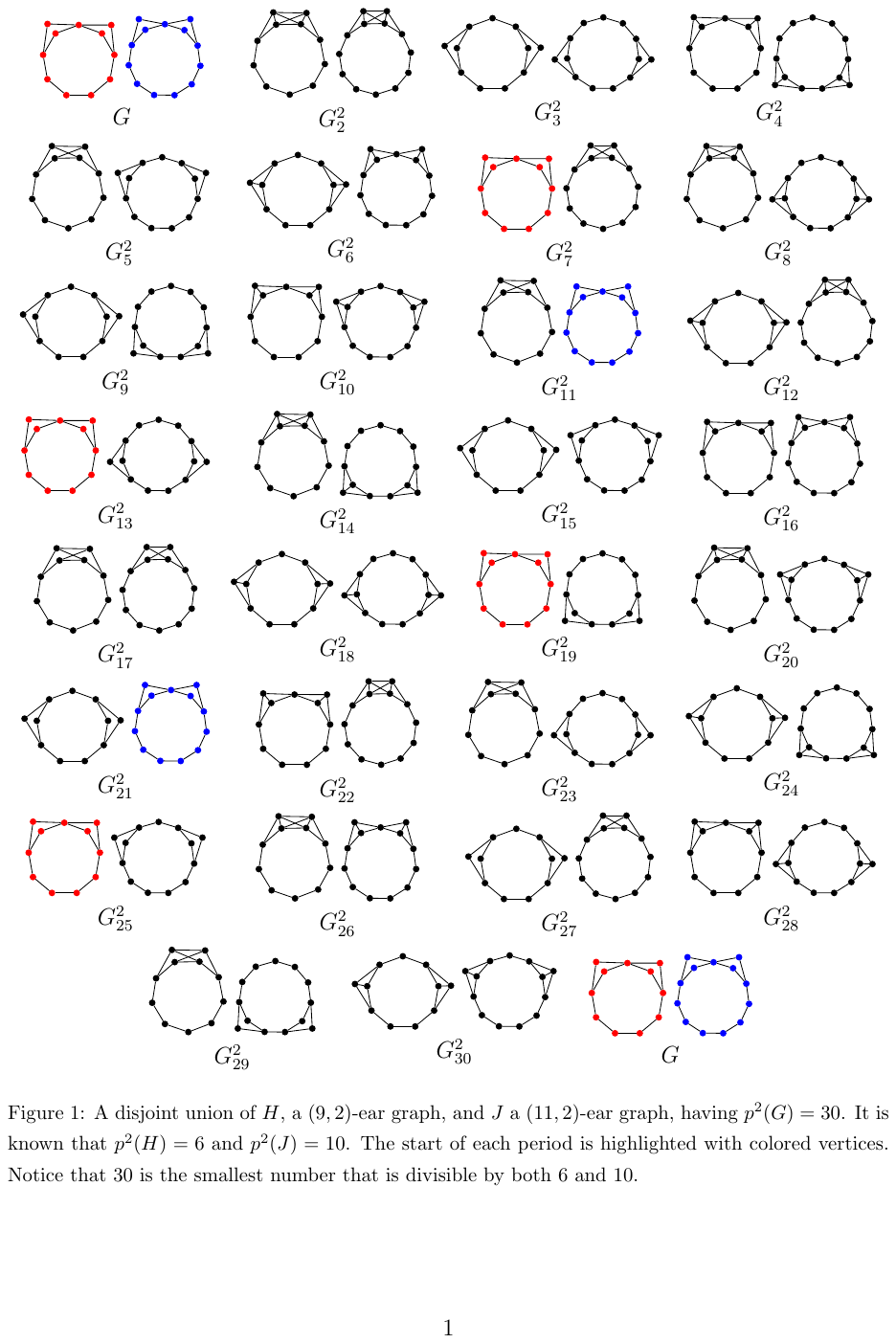}
    \caption{Let $H$ be a $(9,2)$-ear graph and let $J$ be an $(11,2)$-ear graph. Note that $\rho^2(H)=6$ and $\rho^2(J)=10$. Let $G$ be a disjoint union of $H$ and $J$. The above diagram shows that $\rho^2(G) = 30$. For clarity, the red graphs are isomorphic to $H$, thus denoting the start of a strong period of $H$. Similarly, the blue graphs are isomorphic to $J$.}
    \label{fig:disconnectCounter}
\end{figure}

\begin{theorem}
Let $k > 2$ and $p > 1$. A $(k^p+1,1)$-ear has (strong) periodicity $p$ under $k$-distance.\label{thm:StrongAny}
\end{theorem}

\begin{proof}
Formally, we will prove there exists a $G$ satisfying the conclusion of the theorem. By Theorem \ref{thm:omniWeak}, there exists a cycle $C = C_{k^p+1} = (V,E)$ that has $\varrho^k(C) = p$ for any $p > 1$. Of course, $\rho^k(C) = 1$. Let $F$ be the $(|V|,1)$-ear graph that is formed by adding two false twins to $C$. By Lemma \ref{lem:earStrongIsWeak} where $G := F$ and $C := C$, we have that $F \cong F^k_q$ if and only if $C \equiv C^k_q$ for any $q > 1$. This implies that $\rho^k(F) = \varrho^k(C)$. We set $G := F$ and we are done.
\end{proof}


\begin{conjecture}
    Any connected graph $G = (V,E)$ has $\rho^2(G) \leq |V| - 3$.
\label{con:period} 
\end{conjecture}

The above conjecture holds true when $G$ is an $(n,d)$-ear graph for any $n \leq 200$ and $d = 2$. We confirm this fact using simulation. Fig. \ref{fig:Desmos} shows the results of our computation. However, Conjecture \ref{con:period} is false when we drop connectedness as a sufficient condition. For example, if $G$ is the disjoint union of an $(11,2)$-ear and a $(13,2)$-ear, as shown in Fig. \ref{fig:disconnectCounter}, then $\rho^2(G) > |V|$.



   \bibliographystyle{elsarticle-num}
   \bibliography{bibliography}
   

\end{document}